\documentclass[11pt]{amsart} 

\usepackage{verbatim, latexsym, amssymb, amsmath,color}
\usepackage{epsfig}
\def\R{\mathbb R}
\def\RP{\mathbb {RP}}
\def\CP{\mathbb {CP}}
\def\N{\mathbb N}
\def\Z{\mathbb Z}
\def\C{\mathbb C}

\def\area{\mathrm{area}}

\newtheorem*{conj}{Conjecture}

\newtheorem{thm}{Theorem}[section]

\newtheorem{cor}[thm]{Corollary}

\theoremstyle{remark}

\theoremstyle{definition}

\title{The Willmore conjecture}
\author{Fernando C. Marques and Andr\'e Neves}
\address{Instituto de Matem\'atica Pura e Aplicada (IMPA) \\ Estrada Dona Castorina 110 \\ 22460-320 Rio de Janeiro \\ Brazil}
\email{coda@impa.br}
\address{Imperial College \\ Huxley Building \\ 180 Queen's Gate \\ London SW7 2RH \\ United Kingdom}
\email{a.neves@imperial.ac.uk}
\thanks{The first author was partly supported by CNPq-Brazil and FAPERJ. The second author was partly supported by Marie Curie IRG Grant and ERC Start Grant.}

\begin{document}

\maketitle

\begin{abstract}
The Willmore conjecture, proposed in 1965, concerns the quest to find the best torus of all.  This problem has inspired a lot of mathematics over the years, helping bringing together ideas from subjects like conformal geometry,  partial differential equations, algebraic geometry  and geometric measure theory. 

In this  article we survey the history of the conjecture and  our recent solution  through the min-max approach. We finish with a discussion of some of the many open questions that remain in the field.
\end{abstract}

\section{Introduction}

A central theme in Mathematics and particularly in Geometry has been the search for the optimal representative within a certain class of objects. Partially motivated by this principle, Thomas Willmore started in the $1960$s the quest for the optimal  immersion of a compact surface in three-space. This optimal shape was to be found, presumably, by minimizing a natural energy over all compact surfaces in $\R^3$ of a given topological type. In this survey article we discuss the history and our recent solution of the Willmore conjecture, the problem of determining the best torus among all.

We begin by defining the energy. Recall that the local geometry of a surface in $\R^3$ around a point $p$ is described by the principal curvatures $k_1$ and $k_2$, the maximum and minimum curvatures among all intersections of the surface with perpendicular planes passing through $p$. The classical notions of curvature of a surface are then the {\it mean curvature} $H=(k_1+k_2)/2$ and the {\it Gauss curvature} $K=k_1k_2$.
With the aforementioned question in mind, Willmore associated to every compact surface $\Sigma\subset \R^3$ a quantity now known as the {\em Willmore energy}:
\[W(\Sigma)=\int_\Sigma H^2 d\mu = \int_{\Sigma}\left(\frac{k_1+k_2}{2}\right)^2d\mu,\]
where $d\mu$ stands for the area form of $\Sigma$. 

The Willmore energy is remarkably symmetric. It is invariant under rigid motions and scalings, but less obvious is the fact that it is invariant also under the inversion map $x\mapsto x/|x|^{2}.$ Hence $W(F(\Sigma))=W(\Sigma)$ for any conformal transformation $F$ of three-space. It is interesting that Willmore himself  became aware of this fact only after reading the paper of White \cite{whitej}, many years after he started working on the subject.  As we will   explain later, this conformal invariance was actually known already in the 1920s.

 In applied sciences this energy had already been introduced long ago,  to study vibrating properties of thin plates.  Starting in the $1810$s, Sophie Germain \cite{germain} proposed, as the elastic or bending energy of a thin plate, the integral  with respect to the surface area of an even, symmetric function of the principal curvatures. The Willmore energy is the simplest possible example (excluding the area fuctional). Similar quantities were considered by Poisson \cite{poisson} around the same time. 
 
The Willmore energy had also appeared  in Mathematics in the 1920s through  Blaschke \cite{blaschke} and his student Thomsen \cite{thomsen}, of whose works  Willmore was not initially aware. The school of Blaschke was working under the influence of Felix Klein's Erlangen Program, and they wanted to understand the invariants of surface theory in the presence of an  action of a group of transformations. This of course included the M\"{o}bius group - the conformal group acting on Euclidean space with a point added at infinity.   

A natural map studied by Thomsen in the context of Conformal Geometry, named the {\it conformal Gauss map} by Bryant (who rediscovered it in \cite{bryant2}), associates to every point of an oriented compact surface in $\R^3$ its central sphere, the unique oriented sphere having the same normal and mean curvature as the surface at the point. This is a conformal analogue of the Gauss map, that associates to every point of an oriented surface its unit normal vector.  
If we include planes as spheres of mean curvature zero, the space of  oriented spheres in $\R^3$, denoted here by $\mathcal Q$, can be identified with the 4-dimensional Lorentzian unit sphere in the five-dimensional Minkowski space. The basic principle states  that applying a conformal map to the surface corresponds to applying an isometry of $\mathcal Q$ to its conformal Gauss map. Since  the area of the image of a closed surface under the conformal Gauss map is  equal to the Willmore energy of the surface minus the topological constant $2\pi\chi(\Sigma)$, the conformal invariance of ${W}(\Sigma)$ becomes apparent this way.


Finally, the Willmore energy  has continued to appear in applied fields, like in biology to study the elasticity of cell membranes (it is the highest order term in  the Helfrich model \cite{helfrich}), in computer graphics in order to study surface fairing \cite{lott}, or in geometric modeling \cite{botsch}.

Back to the quest for the best  immersion, Willmore showed that round spheres have the least possible Willmore energy among all compact surfaces in three-space.  More precisely,  he proved that every compact surface $\Sigma\subset \R^3$ satisfies
$$W(\Sigma)\geq 4\pi,$$ with equality only for  round spheres.

Here is a geometric way to see this. First note that when a plane is translated from very far away and touches the surface for the first time, it will do so tangentially in a point where the principal curvatures share the same sign. At such points the Gauss curvature $K=k_1k_2$ must necessarily be nonnegative. Therefore the image of the set of points where $K\geq 0$ under the Gauss map $N$ must be the whole $S^2$. Since $K=\det(dN)$,  the area formula tells us that $\int_{\{K\geq 0\}}K d\mu\geq {\rm area}(S^2)=4\pi$. The result of Willmore follows because $H^2 \geq K$, with equality at umbilical points ($k_1=k_2$), and the only totally umbilical surfaces of $\R^3$ are the round spheres.

Willmore continued the study of his energy and, having found the compact surface with least possible energy,  tried to find the minimizing shape among the class of tori  \cite{willmore}. It is interesting to note that no obvious candidate stands out a priori. To develop intuition about the problem, Willmore considered a special type of torus: he  fixed a circle $C$ of radius $R$ in a plane and looked at the tube $\Sigma_r$ of constant radius $r<R$ around $C$. When $r$ is small, $\Sigma_r$ is very thin and the energy 
 $W(\Sigma_r)$ is very large. If we keep increasing the value of $r$, the size of the middle hole of the torus decreases and  eventually the hole disappears for  $r=R$. The energy $W(\Sigma_r)$ becomes  arbitrarily large when $r$ approaches  $R$. Therefore the function $r \mapsto W(\Sigma_r)$  must have an absolute minimum in the interval $(0,R)$, which Willmore computed to be $2\pi^2$. 

Up to scaling, the optimal tube in this class has $R=\sqrt{2}$ and $r=1$: 
$$
\Sigma_{\sqrt{2}}=\{\big( (\sqrt{2} +\cos\, u) \cos\, v, (\sqrt{2}+\cos\,u)\sin\,v, \sin\,u) \in \mathbb{R}^3: u, v \in \mathbb{R}\}.
$$
In light of his findings, Willmore conjectured that this  torus of revolution should minimize the Willmore energy among all tori in three-space:
\begin{conj}[Willmore, 1965] Every compact surface $\Sigma$ of genus one in $\mathbb{R}^3$ must satisfy
$$W(\Sigma)\geq 2\pi^2.$$
\end{conj}

It seems at first rather daring to have made such a conjecture after having it tested only for a very particular one-parameter family of tori. On the other hand, the torus that Willmore found is very special and had already appeared in geometry in disguised form.  It turns out that there exists   a stereographic projection  from the unit $3$-sphere $S^3\subset \mathbb{R}^4$ minus a point onto Euclidean space $\mathbb{R}^3$ that maps the Clifford torus $\hat{\Sigma}=S^1(\frac{1}{\sqrt 2})\times S^1(\frac{1}{\sqrt 2})$ onto $\Sigma_{\sqrt{2}}$. We will say more about the Clifford torus later.

In \cite{marques-neves}, we proved:
\begin{thm}\label{main.thm} Every embedded compact surface $\Sigma$ in $\R^3$ with positive genus satisfies
$$W(\Sigma)\geq 2\pi^2.$$
Up to rigid motions, the equality holds only  for  stereographic projections  of the Clifford torus (like $\Sigma_{\sqrt{2}}$).
\end{thm}
The rigidity statement characterizing the equality case in Theorem \ref{main.thm} is optimal because stereographic projections are conformal and, as we have mentioned, the Willmore energy is conformally invariant.

Since Li and Yau \cite{li-yau} had proven in the $1980$s that compact surfaces with self-intersections have Willmore energy greater than or equal to $8\pi$, our result implies:
\begin{cor}The Willmore conjecture holds.
\end{cor}

\section{Some particular cases and related  results}\label{partial}

In this section we  survey  some  of the  results and techniques that have been used  in understanding the Willmore conjecture.  Our emphasis will be in giving a glimpse of the several  different ideas and approaches that have been used, instead of giving an exhaustive account of every result proven.

The richness of the problem derives partially from the fact that the Willmore energy is invariant under conformal maps. Since stereographic projections are conformal, one immediate consequence is that the conjecture can be restated for surfaces in the unit $3$-sphere $S^3$. Indeed, if $\Sigma$ is a compact surface in $S^3$ and $ \tilde \Sigma$ denotes its image in $\R^3$ under a stereographic projection, then one has
$$ W(\tilde \Sigma)=\int_{\Sigma}1+ \left(\frac{ k_1+ k_2}{2}\right)^2d\mu,$$
where $ k_1,  k_2$ are the principal curvatures of $\Sigma$  with respect to the standard metric on $S^3$. For this reason we take the right-hand side of the equation above as the definition of the Willmore energy   of $\Sigma \subset S^3$: $$\mathcal W(\Sigma)=\int_\Sigma (1+H^2) \,d\mu.$$

The Willmore conjecture can then be restated equivalently in the following form: 
\begin{conj}[Willmore, 1965] Every compact surface $\Sigma\subset S^3$ of genus one must satisfy
$$\mathcal W(\Sigma)\geq 2\pi^2.$$
\end{conj}

We need to introduce  an important  class of surfaces that plays  a central role in understanding the Willmore energy, as we will see later. This is the class of {\em minimal surfaces}, 
defined variationally as surfaces that are stationary configurations for the area functional, i.e., those for which the first  derivative of the area is zero with respect to any variation. 

These surfaces are characterized by the property that their mean curvature $H$ vanishes identically.
Hence  it follows immediately  from the expression for $\mathcal W(\Sigma)$ that the Willmore energy of a minimal surface in $S^3$ is equal to its area. The equators (or great spheres)  are the simplest examples of minimal surfaces in $S^3$, with area $4\pi$, while the Clifford torus  is a minimal  surface in $S^3$ with area $2\pi^2$. Note that this is compatible with the fact that $\Sigma_{\sqrt{2}}$ is a stereographic projection of the Clifford torus and $W(\Sigma_{\sqrt{2}})=2\pi^2$.  There are infinitely many known compact minimal surfaces in $S^3$. For instance Lawson \cite{lawson70} in the $1970$s found embedded orientable minimal surfaces  in $S^3$ of any genus. 

Given that the Willmore conjecture was initially  tested only for a very particular set of tori in $\R^3$, the first wave of results consisted in testing the conjecture on larger classes.  Willmore himself  in $1971$ \cite{willmore71}, and independently Shiohama and Takagi \cite{shiohama-takagi}, verified the conjecture   for tubes of constant radius around a space curve $\gamma$ in $\R^3$. An explicit computation gives that such a torus must satisfy
$$ W(\Sigma)\geq \pi\int_{\gamma}|k|\,ds, $$
where $|k|\geq 0$ is the curvature of the space curve $\gamma$. Hence the result follows from  Fenchel's Theorem \cite{fenchel}, which establishes $\int_{\gamma}|k| \,ds\geq 2\pi.$

In 1973, Chen \cite{chen} checked that every intrinsically flat torus in $S^3$ has Willmore energy  greater than or equal to $2\pi^2$, with equality only for the Clifford torus. The inverse image  under the Hopf map $S^3 \rightarrow S^2$  of a closed curve in $S^2$  is a flat torus in $S^3$ and thus such examples abound. Chen's Theorem  follows from integral geometric arguments that we quickly describe. Given a surface $\Sigma$ in $\R^{k+2}$ with unit normal bundle $B$, we have the generalized Gauss map 
$$\mathcal G:B\rightarrow S^{k+1}=\{v\in \R^{k+2}: |v|=1\}.$$
The {\em total curvature} $\tau_k(\Sigma)$  is defined as the total $(k+1)$-volume parametrized by $\mathcal G$ divided by the volume of $S^{k+1}$.
Chern-Lashof \cite{chern} showed that the total curvature is equal to the average  number of critical points of a linear function on $\Sigma$. Therefore every torus has total curvature greater than or equal to $4$. Through  purely  local computations, Chen showed that if a torus $\Sigma\subset S^3\subset \R^4$ is flat, then 
$$\mathcal W(\Sigma)\geq  \frac{\pi^2}{2}\tau_2(\Sigma),$$ 
that combined with the Chern-Lashof inequality implies the result.  Later he extended  this result to include flat tori in the unit $n$-sphere $S^n$ (\cite{chen2}).  See \cite{ammann} and \cite{topping} for other related results.

In $1976$, Langevin and Rosenberg \cite{langevin-rosenberg} showed that the Willmore energy of a knotted torus in $\R^3$ is at least  $8\pi$. The key estimate was to show that if the torus is knotted  then every linear function has at least $8$ critical points. Hence
$$\frac{1}{4\pi}\int_{\Sigma}|K|d\mu=\frac{1}{2}\tau_1(\Sigma)\geq 4,$$
where the first identity follows from the definition of total curvature and the fact that $K=\det(dN)$, while the inequality follows from the  Chern-Lashof result. Therefore, since $\int_\Sigma K d\mu=0$ by  the Gauss--Bonnet Theorem, we have
$$ W(\Sigma)\geq\int_{\{K\geq 0\}} \left(\frac{k_1+k_2}{2}\right)^2d\mu\geq \int_{\{K\geq 0\}}Kd\mu
=\frac{1}{2}\int_{\Sigma}|K|d\mu\geq 8\pi.
$$
This result is reminiscent of the Fary-Milnor Theorem \cite{fary,milnor}, which states that  the total curvature $\int_C |k|\, ds$ of a knotted closed curve $C$ in $\R^3$ must exceed $4\pi$.

In 1978, Weiner  \cite{weiner} checked that the Clifford torus is a critical point with non-negative second variation of the Willmore energy. If this were not the case there would be some small perturbation of the Clifford torus with strictly smaller energy, contradicting the conjecture.

In $1982$,  Li and Yau \cite{li-yau} were the first to exploit in a crucial way the conformal invariance of the problem. They introduced the important notion of {\it conformal volume} of an immersion $\phi: \Sigma \rightarrow S^n$:
$$
V_c(n,\phi)=\sup_{g \in Conf(S^n)} {\rm area}(g\circ \phi),
$$
  and obtained various striking results. Their ideas, that we describe below, had a lasting impact in Geometry. The results are valid when the ambient space is the unit $n$-sphere but we restrict ourselves to $S^3$ for simplicity. 

The conformal group of $S^3$, modulo isometries, is parametrized by the unit $4$-ball $B^4$:  for each $v\in B^4$ we associate  the conformal map
\begin{equation}\label{conformal.maps}
F_v:S^3 \rightarrow S^3, \quad F_v(x) = \frac{(1-|v|^2)}{|x-v|^2}(x-v) -v.
\end{equation}
 The map $F_0$ is the identity and, for $v\neq 0$, $F_v$ is a conformal dilation with fixed points $v/|v|$ and $-v/|v|$. 
 
 For illustrative purposes, we note  that if $B$ is a geodesic ball in $S^3$ and $p\in S^3$, then  $F_{tp}(B)$ is a geodesic ball that,  as $t<1$ tends to $1$, could have three types of behavior. It converges   to the whole  $3$-sphere if $p$ is inside $B$ , to the antipodal point $-p$ if $p$ is outside the closure of $B$, or to the  hemisphere touching $\partial B$ tangentially at $p$ in case $p$ is in the boundary of $B$.
 
 Using these transformations, Li and Yau showed that if a surface $\Sigma$ contains  a $k$-point  $p$, i.e.,  if nearby $p$  the surface   looks like $k$ small discs containing $p$, then $\mathcal W(\Sigma)\geq 4\pi k.$  The proof goes as follows: as $t<1$ tends to $1$,  $F_{tp}(\Sigma)$ converges to a union of $k$ great spheres (boundaries of hemispheres). Since
 $$
 \mathcal W(\Sigma)=\mathcal W(F_{tp}(\Sigma))\geq {\rm area}(F_{tp}(\Sigma)),
 $$ by taking the limit as $t\rightarrow 1$ we get 
 \[\mathcal W(\Sigma)\geq {\rm area}(k\mbox{ great spheres} )=4\pi k.\]

One consequence of this result is that the energy of every compact surface is at least $4\pi$ (a result we already mentioned). It also follows that the energy of any compact surface that is not embedded, hence contains a double-point at least, must be greater than or equal to $8\pi$. In particular,  one can restrict to the class of  embedded tori in order to prove the Willmore conjecture.

Another consequence is that every immersed projective plane in $\R^3$ must have Willmore energy greater than 
or equal to $12\pi$. This is because such projective planes always contain   a triple-point at least. Bryant \cite{bryant} and Kusner \cite{kusner} found projective planes in $\R^3$ with Willmore energy exactly equal to $12\pi$. All such projective planes were classified   in \cite{bryant}.

In that same paper, Li and Yau also found a region $R$ (see below for an explicit description) in the space of all conformal structures so that if the conformal class of a torus $\Sigma\subset S^3$ lies in $R$, then $\mathcal  W(\Sigma)\geq 2\pi^2$. The conformal class of the Clifford torus (square lattice) is in the boundary of $R$.
This result relates in an ingenious way the  conformal invariance of the Dirichlet energy $\int_\Sigma |\nabla f|^2\, d\mu$ in dimension two with the conformal invariance of the Willmore energy. We discuss briefly  the main ideas.

Given a torus $\Sigma\subset S^3$, the Uniformization Theorem tells us that {$\Sigma$ is diffeomorphic to $\R^2/ \Gamma$, where $\Gamma$ is some lattice in $\R^2$, and that }the induced metric {$g$ on $\Sigma$ is  conformal to the Euclidean metric  $g_0$ on $\R^2/ \Gamma$}.  The lattice $\Gamma$ can be chosen to be generated by  the vectors $(1,0)$ and $(x,y)$, with $0\leq x\leq 1/2$, $y\geq 0$ and $x^2 + y^2 \geq 1$.

Recall that the first nontrivial eigenvalue of the Laplacian  {with respect to $g_0$} is given by
\begin{equation}\label{lambda}
\lambda_1 = \inf_{\int_\Sigma f \, d\mu_{g_0}=0, f\neq 0} \frac{\int_\Sigma |\nabla_{g_0} f|^2\, d\mu_{g_0}}{\int_\Sigma f^2 \, d\mu_{g_0}}.
\end{equation}
Li and Yau first prove, through a degree argument, that $\Sigma$ can be balanced by applying a conformal transformation, i.e. that there exists $v_0\in B^4$ such that $\int_\Sigma x_i \circ F_{v_0} \, d\mu_{g_0}=0$ for every $i=1,\dots,4$.
By evaluating the quotient in the right hand side of equation \eqref{lambda} with {$f= x_i \circ F_{v_0}$}, summing over $i=1,\dots,4$ and using the conformal invariance of the Dirichlet energy in dimension two, they show that
{\begin{align*}
\lambda_1 {\rm area}(\R^2/ \Gamma, g_0)& =\lambda_1 \sum_{i=1}^4 \int_\Sigma (x_i \circ F_{v_0})^2 \, d\mu_{g_0}\leq  \sum_{i=1}^4 \int_\Sigma |\nabla_{g_0}(x_i \circ F_{v_0})|^2 \, d\mu_{g_0} \\
 &=\sum_{i=1}^4 \int_\Sigma |\nabla_{g}(x_i \circ F_{v_0})|^2 \, d\mu_{g}=2 \,{\rm area}(F_{v_0}(\Sigma)).
\end{align*}}
Using the conformal invariance of the Willmore energy we have,
$${{\rm area}(F_{v_0}(\Sigma)) \leq \mathcal W(F_{v_0}(\Sigma))=\mathcal W(\Sigma)},$$
where the first inequality comes from the expression of the Willmore energy in $S^3$. Putting these two inequalities together Li and Yau obtained that
$${\lambda_1 {\rm area}(\R^2/ \Gamma, g_0)\leq 2\,\mathcal W(\Sigma)}.$$
The left-hand side of the above inequality can be computed for every conformal class of the torus. It turns out that if $\Gamma$ is in the    set $R$ of  lattices  such that the generators $(1,0)$ and $(x,y)$ satisfy the additional assumption $y\leq 1$, besides 
$0\leq x\leq 1/2$, $y\geq 0$ and $x^2 + y^2 \geq 1$, then $$4\pi^2\leq \lambda_1 {\rm area}(\R^2/ \Gamma, g_0).$$
This finishes  Li-Yau's proof that $\mathcal{W}(\Sigma) \geq 2\pi^2$ when the conformal class of $\Sigma$ lies in $R$. In 1986, Montiel and Ros \cite{montiel-ros}   found a larger set of lattices, still containing the square lattice on the boundary,  for which the Willmore conjecture holds. 


In $1984$, Langer and Singer \cite{langer-singer} showed that the energy of every torus of revolution (with possibly noncircular section) in space is  greater than or equal to $2\pi^2$, with equality only for the Clifford torus and dilations of it. The basic fact, observed independently by Bryant and Pinkall, is that if $\gamma$ is a closed curve in the upper half-plane $P=\{(x,y,0):y>0\}$, and $\Sigma$ is the torus obtained by revolving $\gamma$ around the $x$-axis,  then $$W(\Sigma)=\frac{\pi}{2}\int_{\gamma}k_{-1}^2ds,$$
 where $k_{-1}$ is the geodesic curvature of $\gamma$ computed with respect to the {\em hyperbolic metric} on $P$.  Langer and Singer showed that  every regular closed curve in the hyperbolic plane satisfies $\int_{\gamma}k_{-1}^2ds\geq 4\pi$, and  the inequality follows. Later, Hertrich-Jeromin and Pinkall \cite{pinkall2} extended this computation to a larger class of tori ({\em Kanaltori} in German), for which the Willmore energy is still given by a line integral.

The variational problem associated to the Willmore energy is extremely interesting, and many solutions are known that are not of minimizing type. Surfaces that constitute critical points of the Willmore energy are called {\em Willmore surfaces} (they were previously called {\it conformal minimal surfaces} by Blaschke). The Euler-Lagrange equation 
for this variational problem, attributed by Thomsen \cite{thomsen} to Schadow, is of fourth order and is the same for surfaces in $\R^3$ or $S^3$:
$$\Delta H + \frac{(k_1-k_2)^2}{2}H=0.$$
  In particular, minimal surfaces are always  Willmore and therefore Lawson's minimal surfaces in $S^3$ provide examples of compact embedded surfaces of any genus that are stationary for the Willmore functional.

 A classification of all Willmore spheres was achieved in a remarkable work of  Bryant \cite{bryant2, bryant}, who  exploited with significant creativity the conformal invariance of the problem. He discovered that their Willmore energies are always of the form $4\pi k$, with $k\in \N\setminus\{2, 3, 5,7\}$ (see also \cite{peng}), and that the only embedded ones are the round spheres. 
 
 This result had a great impact,   so we briefly describe it. Motivated by his finding that the conformal Gauss map of a  Willmore surface is a conformal harmonic map, Bryant constructed a quartic holomorphic differential on any such surface. This holomorphic differential must  vanish on a topological sphere. Bryant used this fact to show that a Willmore sphere must be the conformal inversion of some minimal surface in $\R^3$ with finite total curvature and embedded ends. It follows from the theory of these minimal surfaces that the Willmore energy of the sphere has to be a multiple of $4\pi$.  Finally, Bryant reduced the problem of  finding all possible such minimal immersions to a problem in algebraic geometry concerning zeros and poles of meromorphic maps on $S^2$, from which he derived his classification.  Ejiri, Montiel, and Musso \cite{ejiri, montiel, musso} extended  this work and classified Willmore spheres in $S^4$.

Up to this point, every known Willmore surface was the conformal image of some minimal surface in $\R^3$ or $S^3$. Pinkall \cite{pinkall} found in $1985$ the first examples of embedded Willmore tori that are not of this type.  His idea was to look at the inverse image under the Hopf map $\pi$ of closed curves $\gamma$ in $S^2$. The Willmore energy is given by
$$\mathcal W(\pi^{-1}(\gamma))=\pi\int_{\gamma}1+k^2\, ds,$$
where $k$ is the geodesic curvature of $\gamma$. Langer and Singer \cite{langer-singer2} had shown that there are infinitely many simple closed curves that are critical points of the  functional in the right-hand side (these are called elastic curves). Pinkall argued that the inverse image of an elastic curve is a  Willmore torus,  and that among those the only one that is  conformal to a minimal surface in $S^3$  is the Clifford torus (the inverse image of an equator in $S^2$).  Finally, if any of the Pinkall tori were to be  the conformal image of some minimal surface $S$ in $\R^3$, the embeddedness of the torus would imply that  $S$ had to be a nonplanar minimal surface asymptotic to a plane. These surfaces  do not exist. Later, Ferus and Pedit \cite{ferus-pedit} found more examples of Willmore tori.

In $1991$, the biologists Bensimon and Mutz \cite{bensimon} (see also \cite{bensimon2}), experimentally verified  the Willmore conjecture with the aide of a microscope while studying the physics of membranes!  They produced toroidal vesicles in a laboratory and observed that their  shape, which according to the Helfrich model should approach the minimizer for the Wilmore energy, was the one predicted by Willmore or one of its conformal images.

The existence of a torus that minimizes the Willmore energy among all tori was proven by Simon \cite{simon93} in $1993$. This result was obtained through a technical tour de force and many of the  ideas involved are now widely used in Geometric Analysis. Very briefly, Simon picked a sequence of tori whose energies converge to the least possible value and showed the existence of  a limit in some weak measure theoretic sense. Exploiting with great effectiveness the fact that the  tori in the sequence are embedded (otherwise the energy would be at least $8\pi$), he obtained that the weak limit must be  a smooth embedded surface. A serious difficulty in accomplishing this comes from the conformal invariance of the problem. For instance, one could start with  a minimizing torus  and apply a sequence of conformal maps so that the images look like  some round sphere with increasingly smaller handles attached. In the limit one would obtain a  round sphere  instead of a torus. To overcome this, Simon showed that every torus can be corrected by applying  a carefully chosen conformal map so that it becomes far away in Hausdorff distance from all round spheres. This way he is sure that in his minimization process he will get a limiting surface that is a torus.

More generally, let $\beta_g$ denote the infimum of the Willmore energy among all  orientable compact surfaces of genus $g$.  It was independently observed by Kusner and Pinkall  (see \cite{kuhnel-pinkall}, \cite{kusner89}, \cite{simon93}) that  a Lawson minimal surface of genus $g$, for every $g$, has area strictly smaller than $8\pi$. This implies   $\beta_g<8\pi$. Later Kuwert, Li and Sch\"{a}tzle \cite{kuwert-li-schatzle} showed that $\beta_g$ tends to $8\pi$ as $g$ tends to infinity.

It is natural to ask whether there exists a  genus $g$ surface with energy $\beta_g$. 
Note that for every partition $g=g_1+\ldots+g_k$ by integers $g_i\geq 1$, one has 
$$\beta_g\leq \beta_{g_1}+\ldots+\beta_{g_k}-4(k-1)\pi.$$
In order to see this take $\Sigma_i$ to be a surface of genus $g_i$ with energy arbitrarily close to $\beta_{g_i}$, $i=1,\ldots,k$. Apply a conformal map to $\Sigma_i$ to get a surface $\tilde{\Sigma}_i$ that can be decomposed into two regions: one that looks like a round sphere of radius one minus a small spherical cap and another one that contains $g_i$ small handles. Remove the handle regions for $i\geq 2$ and sew them into $\tilde{\Sigma}_1$ to get a surface with $g$ handles and energy close to $\beta_{g_1}+\ldots+\beta_{g_k}-4(k-1)\pi$. Simon \cite{simon93}  showed that $\beta_g$ is  indeed attained provided there is no partition $g=g_1+\ldots+g_k$ by integers $g_i\geq 1$ with $k\geq 2$ such that each $\beta_{g_i}$ is attained and
\begin{equation}\label{douglas}
\beta_g=\beta_{g_1}+\ldots+\beta_{g_k}-4(k-1)\pi.
\end{equation}

Bauer and Kuwert \cite{bauer-kuwert}, inspired by Kusner \cite{kusner96}, showed that such partitions do not exist, hence
$\beta_g$ can be realized by a surface of genus $g$ for all $g$. More precisely, they show that if $M$ with genus $h$ and $N$ with genus $s$ realize $\beta_h$ and $\beta_s$, respectively, then  a careful connected sum near  non-umbilic points of $M$ and $N$ produces a surface with genus $g=h+s$ and Willmore energy {\em stricly} smaller than $\beta_h+\beta_s-4\pi$. This suffices to show that \eqref{douglas} never occurs.


Incidentally, one immediate consequence of  our Theorem \ref{main.thm} is that $\beta_g\geq 2\pi^2$ for all $g\geq 1$. Since  we also have $\beta_g<8\pi$, it is not difficult to see that partitions as above can never occur. 


In 1999, Ros \cite{ros} proved the Willmore conjecture for tori in $S^3$ that are preserved by the antipodal map. (This also follows from the work of Topping \cite{topping} on integral geometry.) More precisely, Ros showed that every orientable surface in  the projective space $\RP^3$ has Willmore energy greater than or equal to $\pi^2$.  His approach was quite elegant and based on the fact, proven by Ritor\'{e} and Ros \cite{ritore-ros}, that  of all  surfaces which  divide $\RP^3$ into two pieces of the same volume, the Clifford torus is the one with least area (which in this case is $\pi^2$). The result follows because he also proved that the surface induced in $\RP^3$ by an antipodally-symmetric surface of odd genus in $S^3$ is necessarily orientable.

We now describe the method of Ros because it was inspirational in our approach. An orientable surface $\Sigma$ in $\RP^3$  must be the boundary of a region $\Omega$, which we can choose so  that the volume of $\Omega$ is not bigger than half the volume of $\RP^3$. Next, Ros looked at the region $\Omega_t$ of points that are at a  distance at most $t$ from $\Omega$. For very large  $t$,  $\Omega_t$ will be the whole $\RP^3$ and thus there must exist some $t_0$ so that the volume of $\Omega_{t_0}$ is equal to half the total volume. The result we mentioned in the previous paragraph implies that ${\rm area}(\partial \Omega_{t_0})\geq \pi^2$. Finally, Ros observed that
\begin{equation}\label{heintze}
\mathcal W(\Sigma)\geq {\rm area}(\partial \Omega_t)\quad\mbox{for all }t\geq 0
\end{equation}
and thus $\mathcal W(\Sigma)\geq \pi^2$, as he wanted to show.  The above inequality is a consequence of more general inequalities due to Heintze and Karcher \cite{heintze-karcher}, and it plays a crucial role in our method  as well.

We briefly sketch its proof. If $N$ is the normal vector of $\Sigma$ that points outside $\Omega$, we can consider the map $$\psi_t:\Sigma\rightarrow \RP^3, \quad \psi_t([x])=[ \cos tx+\sin t N(x)].$$
This map is well defined independently of the representative, $x$ or $-x$, we choose for $[x]$. Since $\partial \Omega_t\subset \psi_t(\Sigma)$, we have that
$${\rm area}(\partial \Omega_t)\leq \int_{\{{\rm Jac\,}\psi_t\geq 0\}} {\rm Jac }\,\psi_t d\mu.$$
Denoting by $k_1,k_2$ the principal curvatures of $\Sigma$ with principal directions $e_1,e_2$, respectively, so that $D_{e_i}N=-k_ie_i,$ $i=1,2$,  we have 
\begin{align*}
{\rm Jac }\,\psi_t &=(\cos t-\sin tk_1)(\cos t-\sin t k_2)\\
 &=\cos^2 t +\sin^2 tk_1k_2-\sin t\cos t (k_1+k_2)\\
 & \leq \cos^2 t +\frac{1}{4}\sin^2t (k_1+k_2)^2-\sin t\cos t (k_1+k_2)\\
 &= \left(\cos t-\frac{k_1+k_2}{2}\sin t\right)^2\leq 1+ \left(\frac{k_1+k_2}{2}\right)^2.
\end{align*}
Therefore
\begin{multline*}
{\rm area}(\partial \Omega_t) \leq \int_{\{{\rm Jac\,}\psi_t\geq 0\}} {\rm Jac }\,\psi_t d\mu\leq \int_{\{{\rm Jac\,}\psi_t\geq 0\}} 1+ \left(\frac{k_1+k_2}{2}\right)^2 d\mu
\leq \mathcal W(\Sigma).
\end{multline*}

One year later, Ros \cite{ros2001} used again the area inequality \eqref{heintze}  to show that any odd genus surface in $S^3$ invariant  under the mapping $(x_1,x_2,x_3,x_4)\mapsto (-x_1,-x_2,-x_3, x_4)$ must have Willmore energy greater than or equal to  $2\pi^2$. 


Finally, we mention that the understanding of the analytical aspects of the Willmore equation has  been greatly improved in recent years thanks primarily to the works of Kuwert, Sch\"{a}tzle (e.g. \cite{kuwert-schatzle}) 
and Rivi\`ere
(e.g. \cite{riviere}). In \cite{kuwert-schatzle},  Kuwert and Sch\"{a}tzle  analyze isolated singularities of Willmore surfaces in codimension one and prove, as a consequence, that the Willmore flow (the $L^2$ negative gradient flow of the Willmore energy) of a sphere with energy less than $8\pi$ exists for all time and becomes round. This uses a blow-up analysis of possible singular behavior of the flow and Bryant's classification of Willmore spheres in codimension one.  In \cite{riviere}, Rivi\`{e}re derives a general weak formulation of the Willmore Euler-Lagrange equation in divergence form, in any codimension,  and proves regularity of weak solutions. 
He also extends the analysis of point singularities of \cite{kuwert-schatzle} to the higher codimension case.   Here is a list of some other current topics of interest, on which there is a lot of activity today: compactness results (modulo the M\"{o}bius group) and blow-up analysis of  Willmore surfaces, existence of minimizers of the Willmore functional under different constraints (fixed topology, conformal class, isoperimetric ratio), study of Willmore-type functionals in Riemannian manifolds among others. This is a fast paced area with many developments currently taking place and so, instead of listing them exhaustively, we point the reader to the   survey article on the Willmore functional of Kuwert and Sch\"{a}tzle \cite{kuwert-schatzle-survey},  and to the lecture notes of Rivi\`ere \cite{riviere-notes} where an introduction to the analysis of conformally invariant variational problems is provided.

\section{Min-max approach}\label{approach}
We now describe the approach we used to solve the Willmore conjecture. As in the previous sections, we appeal to intuition so that we can emphasize the geometric nature of the arguments.


One purpose of the min-max technique is  to find unstable critical points of  a given functional, i.e., critical points that are not of minimum type. 
For example, consider the surface $M=\{(x,y,z)\in \R^3: z=x^2-y^2\}$ and  the height function $f(x,y,z)=z$. Then $(x,y,z)=(0,0,0)$ is a critical point of $f$ of saddle type. It turns out that it is possible to detect this critical point by a variational approach. 
The idea is to  fix a continuous path $\gamma:[0,1]\rightarrow M$  connecting $(0,1,-1)$ to $(0,-1,-1)$ and set $[\gamma]$ to be the collection of all continuous paths $\sigma:[0,1]\rightarrow M$ that are homotopic to  $\gamma$ with fixed endpoints. We  define 
$$L=\inf_{\sigma\in [\gamma]}\max_{0\leq t\leq 1} f(\sigma(t)).$$

The projection on the $xy$-plane of any path  in $[\gamma]$  has to intersect the diagonal  line $\{(t,t,0)\in \R^3:t\in \R\}$, where $f$ vanishes.  Hence $L\geq 0$. But considering $\sigma(t)=(0,1-2t,-(1-2t)^2)$, $0\leq t\leq 1$, we see that $$0=f(0,0,0)=f(\sigma(1/2))=\max_{0\leq t\leq 1} f(\sigma(t)).$$
Hence $L=0$.
The tangential projection of $\nabla f_{|(0,0,0)}$  on $M$  has to vanish  because otherwise we would be able to perturb the path $\sigma$ near the origin and obtain another path $\bar\sigma$ in $[\gamma]$ with $\max_{0\leq t\leq 1} f(\bar \sigma(t))<f(0,0,0)=0$. This is impossible because $L=0$, and so we have found a critical point of $f$ restricted to $M$ via a variational method.

Note that the same reasoning implies that the index of the critical point has to be less than or equal to one because otherwise the origin would be a strict local maximum for $f$. Again we would be able to perturb $\sigma$ to obtain another   $\bar\sigma$ in $[\gamma]$ with $\max_{0\leq t\leq 1} f(\bar \sigma(t))<f(0,0,0)=0$.

Almgren \cite{almgren-varifolds} in the $1960$s developed a  Min-max Theory for the area functional.  His motivation was to produce minimal surfaces (which are critical points for the area functional, as we have mentioned before). The techniques he used come from the field of Geometric Measure Theory, but we will try to keep the discussion with as little technical jargon as possible. This min-max theory applies to  more general ambient manifolds,  but we restrict  to the case of  the round $3$-sphere for simplicity.

Denote by $\mathcal Z_2(S^3)$ the space of integral 2-currents with boundary zero, which can be thought of intuitively as the space of oriented Lipschitz closed surfaces in $S^3$, with integer multiplicities. For instance, the boundary $\partial U$ of any open set $U$ of finite perimeter is a well-defined element of $\mathcal Z_2(S^3)$. This space is endowed with the flat topology, according to which two surfaces are close to each other if the volume of the region in between them is very small. 

  Let  $I^k=[0,1]^k$ be the unit $k$-dimensional cube. We will consider  continuous functions $\Phi:I^k\rightarrow \mathcal Z_2(S^3)$, and denote by $[\Phi]$ the set of all continuous maps $\Psi:I^k\rightarrow\mathcal Z_2(S^3)$ that are homotopic to $\Phi$ through homotopies that fix the maps on $\partial I^k$. We then define
$${\bf L}([\Phi])=\inf_{\Psi\in [\Phi]}\sup_{x\in I^k} {\rm area}(\Psi(x)).$$
The prototype theorem was proven by Pitts \cite{pitts}  (see also \cite{colding-delellis} for a nice exposition), a student of Almgren at the time, and states the following.

\begin{thm}[Min-max Theorem] Suppose $${\bf L}([\Phi])>\sup_{x\in \partial I^k}{\rm area}(\Phi(x)).$$ 
Then there exists a disjoint collection of smooth, closed, embedded  minimal  surfaces $\Sigma_1,\dots,\Sigma_N$ in $S^3$  such that
$${\bf L}([\Phi])=\sum_{i=1}^N m_i\,{\rm area}(\Sigma_i),$$
for some positive integer multiplicities $m_1,\dots,m_N$.
\end{thm}

 The condition ${\bf L}([\Phi])>\sup_{x\in \partial I^k}{\rm area}(\Phi(x))$ in the above theorem  is important and reflects the fact that $[\Phi]$ is  capturing  some nontrivial topology of $\mathcal Z_2(S^3)$. One should also expect that the index of $\Sigma$, i.e., the maximum number of linearly independent deformations whose linear combinations all decrease the area of $\Sigma$, should be no higher than $k$. This is because the definition of $ {\bf L}([\Phi)]$ suggests that $\Sigma$ is maximized in at most $k$ directions. Due to the technical nature of the subject, this expectation remains to be proven.  

It is instructive to study a simple example. The equator $S^3\cap \{x_4=0\}$ in $S^3$ is a minimal surface and its index is one. One way to see that its index is greater than or equal to one is to note that  if we move the equator up with constant speed,  the area decreases. On the other hand, any deformation of the equator that fixes the enclosed volume cannot decrease the area because the equator is the least area surface that divides $S^3$ into two pieces of the same volume. Therefore, the index of the equator is exactly one. 

To produce the equator  using the Almgren-Pitts Min-max Theory we consider the map
 $$\Phi_1:[0,1]\rightarrow \mathcal Z_2(S^3), \quad\Phi_1(t)=S^3\cap\{x_4=2t-1\}$$ and the corresponding homotopy class $[\Phi_1]$. Given  $\Psi\in [\Phi_1]$, there is some $0\leq t_0\leq1$ such that  $\Psi(t_0)$ divides $S^3$ in two pieces of identical volume. Hence the area of $\Psi(t_0)$ must be  bigger than or equal to the area of an equator, and this implies  $L([\Phi_1])\geq 4\pi$. Moreover, $$L([\Phi_1])\leq \sup_{0\leq t\leq 1}{\rm area}(\Phi_1(t))= 4\pi$$
 and thus $L([\Phi_1])=4\pi$.


The construction of the equator by  the Almgren-Pitts min-max theory is so natural that we became  interested in the following  question:
\begin{center}
{\em Can we produce the Clifford torus using a min-max method?}
\end{center}

This question is specially suggestive given the following result of Urbano: 
\begin{thm}[\cite{urbano} , 1990]  Let $\Sigma$ be a compact minimal surface of $S^3$ with index no bigger than $5$. Then either $\Sigma$ is an equator (of index one) or a Clifford torus (of index $5$).
 \end{thm}

The proof consists in a very elegant and short argument that we now quickly describe. Denote by $\mathcal L$ the second variation operator associated to the area functional, i.e., if  $N$ is the unit normal vector to $\Sigma$, then for every $f\in C^{\infty}(\Sigma)$ we have
$$\frac{d^2}{(dt)^2}_{|t=0}{\rm area}\left(P_{t}(\Sigma)\right)=- \int_{\Sigma}f\mathcal{L}f\,d\Sigma,$$
where $\{P_t\}_{-\varepsilon<t<\varepsilon}$ is a one parameter family of diffeomorphisms generated by a vector field $X$ that satisfies  $X=fN$ along $\Sigma$.

Let $e_1,\ldots,e_4$ denote the coordinate vectors in $\R^4$. An explicit computation shows that the functions $\langle N,e_i\rangle$, $i=1,\ldots,4$,  are eigenfunctions of $\mathcal L$ on $\Sigma$ with  eigenvalue $\lambda=-2$. One can also see that they are linearly independent,  unless the minimal surface $\Sigma$ is an equator. Moreover,  it is well known that the first eigenfunction has multiplicity one and so  it cannot belong to the span of $\{\langle N,e_1\rangle,\ldots, \langle N,e_4\rangle\}$ (unless again the minimal surface $\Sigma$ is an equator).  Hence, if $\Sigma$ is not an equator then the index of $\Sigma$ is at least $5$. In case it is equal to $5$, Urbano used the Gauss--Bonnet Theorem to show that $\Sigma$ has to be the Clifford torus.

Motivated by this, we defined a $5$-parameter family of surfaces that we explain now after introducing some notation.
Let $F_v$ be the conformal maps defined in \eqref{conformal.maps}.  Given an embedded surface $S=\partial \Omega$, where $\Omega$ is a region of $S^3$,  we denote  by $S_t$ the surface at distance $t$ from $S$, which means that $S_t$ is given by $$\partial\{x\in S^3:d(x,\Omega)\leq t\}\quad\mbox{if} \quad  0\leq t\leq \pi$$ and by $$\partial(S^3\setminus\{x\in S^3:d(x,S^3\setminus\Omega)\leq -t\})\quad\mbox{if}\quad-\pi\leq t<0.$$
Note that  $S_t$ is not necessarily a smooth surface (due to  focal points), but it is nonetheless the boundary of an open set with finite perimeter hence constitutes a  well defined element of $\mathcal Z_2(S^3)$. 

Given an embedded compact surface $\Sigma\subset S^3$, we defined in \cite{marques-neves} the {\em canonical family} $\{\Sigma_{(v,t)}\}_{(v,t)\in B^4\times [-\pi,\pi]}$  of $\Sigma$ by
$$ \Sigma_{(v,t)}=(F_v(\Sigma))_t\in \mathcal Z_2(S^3).$$
If $\Sigma$ is the Clifford torus, then  the infinitesimal deformations of $\Sigma_{(v,t)}$ near $(v,t)=(0,0)$ correspond to the $5$ linearly independent directions described in the proof of Urbano's Theorem.   In light of this, we decided to apply the min-max method  to the  homotopy class of such canonical families  in order to produce the Clifford torus.

The canonical family also has the great property that
$$
\area(\Sigma_{(v,t)})\leq \mathcal W(F_v(\Sigma)) =\mathcal W(\Sigma),\quad\mbox{for all }(v,t)\in  B^4\times [-\pi,\pi].
$$
 The above inequality is just like inequality  \eqref{heintze}, while the identity is a consequence of the conformal invariance of the Willmore energy.

From these ingredients we devised  a strategy to prove the Willmore conjecture. If the homotopy class  $\Pi$, determined by the canonical family associated to a surface with positive genus,   indeed produced the Clifford torus via min-max then we would have from the above inequality that 
$$2\pi^2=L(\Pi)\leq \sup_{(v,t)\in B^4\times [-\pi,\pi]} \area(\Sigma_{(v,t)})\leq \mathcal W(\Sigma).$$ 
At this point the question of whether  or not the canonical family could produce  the Clifford torus by a min-max method was upgraded from an issue on which we had an academic interest  to a question which we {\em really} wanted to answer. 

Hence it became important to  understand the geometric and topological properties of the canonical family, especially the behavior of $\Sigma_{(v,t)}$ as $(v,t)$ approaches the boundary of  the parameter space $B^4\times [-\pi,\pi]$. The fact that the diameter of $S^3$ is equal to $\pi$ implies that $\Sigma_{(v,\pm\pi)}=0$ for all $v\in B^4$. Hence we are left to analyze what happens when $v$ approaches $S^3=\partial B^4$.

Assume $v\in B^4$ converges to $p\in S^3$. If $p$ does not belong to $\Sigma$, then  it should be clear that $F_v(\Sigma)$  is pushed into  $\{-p\}$ as $v$ tends to $p$, and  that ${\rm area}(F_v(\Sigma))$ converges to zero in this process. When $p$ lies in $\Sigma$ the situation is  more subtle. Indeed,  if $v$ approaches $p$ radially, i.e., $v=sp$ with $0<s<1$, then $F_{sp}(\Sigma)$ converges, as $s$ tends to $1$, to the unique great sphere tangent to $\Sigma$ at $p$. Therefore the family of   continuous functions in $S^3$ given by $f_s(p)= {\rm area}\,(\Sigma_{sp})$ converges pointwise, as $s\to 1$, to a discontinuous function that is zero outside $\Sigma$ and $4\pi$ along $\Sigma$. 

Therefore,  for any $0<\alpha<4\pi$ and $p\in \Sigma$, there must exist a sequence $\{v_i\}_{i\in\N}$ in  $B^4$ converging  to $p$ so  that $ {\rm area}(\Sigma_{v_i})$ converges to $\alpha$ and thus it is natural to expect that the convergence of $F_v(\Sigma)$ depends on how $v$ approaches $p\in \Sigma$. A careful analysis revealed that, depending on the angle at which $v$ tends to $p$, $F_v(\Sigma)$ converges to a round sphere tangent to $\Sigma$ at $p$, with radius and center depending on the angle of convergence.

Initially we were somewhat puzzled by this behaviour, but then we realized that, even if this parametrization became discontinuous near the boundary of the parameter space, the closure of the family $\{\Sigma_{(v,t)}\}_{(v,t)\in B^4\times [-\pi,\pi]}$ in $\mathcal Z_2(S^3)$  constituted  a nice continuous $5$-cycle relative to the space of round spheres $\mathcal{G}$. In other words, the discontinuity of the canonical family was being caused by  the parametrization of the conformal maps of $S^3$ chosen in \eqref{conformal.maps}. 

To address this issue we performed a blow-up procedure along the surface $\Sigma$ and we were able to  reparametrize the canonical  family by a continuous map $\Phi:I^5\to \mathcal Z_2(S^3)$. The image  $\Phi(I^5)$  is equal to the closure of $\{\Sigma_{(v,t)}\}_{(v,t)\in B^4\times [-\pi,\pi]}$ in $\mathcal Z_2(S^3)$. Moreover, we have:
\begin{itemize}
\item[(A)] $\sup_{x\in I^5}{\rm area}(\Phi(x))=\sup_{(v,t)\in B^4\times[-\pi,\pi]}\area(\Sigma_{(v,t)})\leq \mathcal W(\Sigma);$
\vskip 0.05in
\item[(B)] $\Phi(x,0)=\Phi(x,1) = 0$ for any $x\in I^4$;
\vskip 0.05in

\item[(C)]  for any $x\in \partial I^4$   there exists $Q(x)\in S^3$ such that $\Phi(x,t)$ is a sphere of radius $\pi t$ centered at $Q(x)$ for every $t\in I$.
\end{itemize} 
The explicit expression for the center map $Q:\partial I^4\rightarrow S^3$  mentioned in property  (C) can be found in \cite{marques-neves}.  

Finally, we discovered a key topological fact: 
\begin{itemize}
\item[(D)] the degree of the center map $Q:S^3\rightarrow S^3$ is  equal to the genus of $\Sigma$. Hence it is nonzero by assumption.
\end{itemize}

This point is absolutely crucial  because it shows  that the topology of the surface $\Sigma$, i.e. the information of its genus, determines topological properties of the map $\Phi:I^5 \rightarrow \mathcal{Z}_2(S^3)$.  Note that if the surface $\Sigma$ we start with is a topological sphere,   our approach could not  work because the Willmore conjecture fails in this case. The genus of the surface $\Sigma$ enters in our approach via property (D). 
 
In \cite{marques-neves}, we proved the following result: 

\begin{thm}[$2\pi^2$ Theorem] Consider a continuous map $\Phi:I^5\to \mathcal Z_2(S^3)$ satisfying properties (B), (C), and (D). Then
$$\sup_{x\in I^5}{\rm area}(\Phi(x))\geq 2\pi^2.$$
\end{thm}
 
We now briefly sketch some of the main ideas behind the proof of the $2\pi^2$ Theorem. The first thing to show is that 
\begin{equation}\label{topo}
L([\Phi])>4\pi=\sup_{x\in \partial I^5} {\rm area}(\Phi(x)).
\end{equation}

The proof is by contradiction, hence assume $L([\Phi])=4\pi$.
Let $\mathcal R$ denote the space of all oriented great spheres. This space is canonically homeomorphic to $S^3$ by identifying a great sphere with its center. With this notation, we have  from condition (C)  that $$\Phi(\partial I^4\times\{1/2\})\subset \mathcal R.$$
Moreover, the degree of the map $\Phi:\partial I^4\times\{1/2\}\rightarrow \mathcal R\approx S^3$ is equal to  ${\rm deg}(Q)$ by property  (D), and thus it is nonzero. For simplicity, suppose we can find $\Psi\in[\Phi]$ so that
 $$\sup_{x\in I^5}{\rm area}(\Psi(x))={\bf L}([\Phi])=4\pi.$$ 

The basic fact is that, given any continuous path $\gamma: [0, 1]\rightarrow I^5$ connecting $I^4 \times \{0\}$ to $I^4 \times \{1\}$, the map $\Psi\circ\gamma:[0,1]\rightarrow \mathcal{Z}_2(S^3)$ is a one-parameter sweepout of $S^3$ with 
$$
\sup_{t\in [0,1]} {\rm area}((\Psi\circ\gamma)(t)) \leq 4\pi.
$$
But $4\pi$ is the optimal area for the one-parameter min-max in $S^3$. Hence 
$$
\sup_{t\in [0,1]} {\rm area}((\Psi\circ\gamma)(t)) =4\pi,
$$
and there must exist some $t_0\in (0,1)$ such that $\Psi(\gamma(t_0))$ is a great sphere, i.e., such that $\Psi(\gamma(t_0))\in \mathcal R$. 
 
 Using this, we argue in \cite{marques-neves} that there should be a $4$-dimensional submanifold $R$ in $I^5$, separating the top from the bottom of the cube, such that $\Psi(R)\subset \mathcal{R}$ and $\partial R=\partial I^4\times \{1/2\}$.
Hence
$$\Psi_{*}[\partial R]=\partial[ \Psi(R)]=0\mbox{ in }H_3(\mathcal R,\Z).$$
On the other hand, $\Psi=\Phi$ on $\partial R=\partial I^4\times \{1/2\}$ and so
$$\Psi_{*}[\partial R]=\Phi_{*}[\partial I^4\times\{1/2\}]={\rm deg}(Q)[\mathcal R]\neq 0.$$
This is a contradiction, hence $L([\Phi])>4\pi$.

Because of this strict inequality, we can invoke the Min-max Theorem and obtain a closed minimal surface $\hat{\Sigma}\subset S^3$, possibly disconnected and with integer multiplicities, such that $L([\Phi])=\area(\hat{\Sigma})$. Since the area of any compact minimal surface in $S^3$ is at least $4\pi$, we  assume  that  $\hat{\Sigma}$ is connected with multiplicity one. Otherwise  $L([\Phi])={\rm area}(\hat{\Sigma})\geq 8\pi>2\pi^2$ and we would be done. 

 It is natural to expect that $\hat{\Sigma}$ has index at most five because $\Phi$ is defined  on a $5$-cube.  Urbano's Theorem would imply in this case that $\hat{\Sigma}$ should be either an equator or a Clifford torus. But since $L([\Phi])=\area(\hat{\Sigma})>4\pi$, the surface  $\hat{\Sigma}$ would have to be a Clifford torus, and then
 $$2\pi^2=\area(\hat{\Sigma})=L([\Phi])\leq\sup_{x\in I^5}{\rm area}(\Phi(x)).$$
Because the  index estimate in the Almgren-Pitts theory is not available, we had to exploit the extra structure coming from the canonical family to get an index estimate. See  \cite{marques-neves} for more details. 

\begin{proof}[Proof of Theorem \ref{main.thm}] We can now put everything together and explain how to prove the inequality in Theorem \ref{main.thm}. Through a stereographic projection, we can think of $\Sigma$ as a compact surface with positive genus  in $S^3$. The canonical family associated to $\Sigma$ gives us a  map $\Phi:I^5\to \mathcal Z_2(S^3)$ satisfying properties (A), (B), (C), and (D). We use the $2\pi^2$ Theorem and property (A) to conclude that
$$2\pi^2\leq\sup_{x\in I^5}{\rm area}(\Phi(x))\leq \mathcal W(\Sigma).$$
In the equality case $\mathcal W(\Sigma)=2\pi^2$, there must exist some $(v_0,t_0)$ such that ${\rm area}(\Sigma_{(v_0,t_0)}) = \mathcal W(\Sigma_{v_0})= 2\pi^2.$ With some extra work (see \cite{marques-neves}) we prove that $t_0=0$ and that $\Sigma_{v_0}$ is the Clifford torus. Since $\Sigma=F_{v_0}^{-1}(\Sigma_{v_0})$ and $F_{v_0}$ is conformal, the rigidity case also follows.
\end{proof}

\section{Beyond the Willmore Conjecture}

The study of the Willmore energy is a beautiful subject which, as we could see in Sections \ref{partial} and  \ref{approach}, has brought together  ideas from conformal geometry, geometric analysis, algebraic geometry, partial differential equations and geometric measure theory. Nonetheless many fundamental questions remain unanswered. We finish by discussing some of them.

A basic question is to determine the minimizing shape among surfaces of genus $g$ in $\R^3$, for $g\geq 2$. A conjecture of Kusner \cite{kusner96} states that the minimizer is $\xi_{1,g}$, a genus $g$ minimal surface found by Lawson \cite{lawson70}. Some numerical evidence for this conjecture was provided in \cite{hsu-kusner-sullivan}. It would be interesting to determine the index of $\xi_{1,2}$. 

Another natural question is whether the Clifford torus also minimizes the Willmore energy among tori in $\R^4$. Li and Yau  \cite{li-yau} showed that the Willmore energy of any immersed $\RP^2$ in $\R^4$ is greater  than or equal to $6\pi$. This value is optimal because of the Veronese surface. It would be nice to have an analogue of Urbano's Theorem in this setting.

For surfaces in $\CP^2$, Montiel and Urbano \cite{montiel-urbano} showed that the quantity $\mathcal W(\Sigma)=\int_{\Sigma}2+|H|^2d\mu$ (the Willmore energy) is conformally invariant. They conjectured  that the Clifford torus  minimizes the Willmore energy among all tori. The Willmore energy of the Clifford torus in $\CP^2$  is equal to $8\pi^2/3\sqrt 3$. For comparison,    the  Willmore energy of a complex projective line is $2\pi$ and there are  totally geodesic projective planes with Willmore energy equal to $4\pi$. Montiel and Urbano also showed that the Willmore energy of complex tori is greater than or equal 
to $6\pi>8\pi^2/3\sqrt 3$, which gives some evidence towards their conjecture. If answered positively, this conjecture would help in finding the  nontrivial Special Lagrangian cone in $\C^3$ with least possible density.

Another interesting problem \cite{kusner89} is to  determine the infimum of the Willmore energy in $\R^3$ or $\R^4$ among all non-orientable surfaces of a given genus {or among  all surfaces in a given regular homotopy class (for instance, in the regular homotopy class of a twisted torus)}.  As far as we know, it is not even known whether the minimum is attained. In a related problem, Kusner \cite{kusner96} conjectured  that  a surface in $\R^4$ with Willmore energy smaller than $6\pi$ has to be a sphere.

The study of the gradient flow of the Willmore energy, referred to as the Willmore flow, also suggests many unanswered  questions. For instance, it is not known whether the flow develops finite time singularities. There is numerical evidence \cite{blatt, mayer} showing that it can  happen.  From Bryant's classification (\cite{bryant2}), we know that the lowest energy of  a non-umbilical immersed Willmore sphere is $16\pi$. Hence one would  expect that Willmore spheres with energy $16\pi$ could be perturbed  so that the flow exists for all time and converges to a round sphere. Lamm and Nguyen classified the singularity models that could potentially arise in this situation \cite{lamm}. 

There is an interesting relation with the process of turning a sphere inside out. The existence of these deformations, called  sphere eversions,  was discovered by Smale long ago \cite{smale-eversion}. The point is that, by a result of Banchoff and Max \cite{banchoff-max}, every sphere eversion must pass through a sphere that has at least one quadruple point. The Willmore energy of such sphere is by Li-Yau's Theorem  greater than or equal to $16\pi$. The existence of a Willmore flow line  connecting a Willmore sphere of energy $16\pi$ to a round sphere would therefore provide an optimal sphere eversion. This possibility was proposed by Kusner and confirmed experimentally in \cite{francis}. 

\vspace{1cm}

{{\bf Acknowledgements:}   The first author is grateful to  \'{E}cole Polytechnique, \'{E}cole Normale Sup\'{e}rieure and Universit\'{e} Paris-Est (Marne-la-Vall\'{e}e)
for the hospitality during the writing of this paper. 

{Both authors are very grateful to Robert Bryant and Robert Kusner for their various remarks and suggestions that improved the exposition of the paper.} }

\bibliographystyle{amsbook}

\begin{thebibliography}{99}






\bibitem{almgren-varifolds}
F. Almgren, \textit{The theory of varifolds.} Mimeographed notes, Princeton (1965).



\bibitem{ammann}
B. Ammann, 
\textit{The Willmore conjecture for immersed tori with small curvature integral.} 
Manuscripta Math. 101 (2000), 1Ð22. 

\bibitem{banchoff-max}
T. Banchoff and N. Max,
\textit{Every sphere eversion has a quadruple point,} Contributions to analysis and geometry (Baltimore, Md., 1980), pp. 191Ð209, Johns Hopkins Univ. Press, Baltimore, Md., 1981. 

\bibitem{bauer-kuwert}
M. Bauer and E. Kuwert, 
\textit{Existence of minimizing Willmore surfaces of prescribed genus,}
Int. Math. Res. Not. (2003), 553--576. 


\bibitem{blaschke}
W. Blaschke,
\textit{Vorlesungen \"Uber Differentialgeometrie III,} Berlin: Springer  (1929).


\bibitem{blatt}S. Blatt, \textit{A singular example for the Willmore flow,} Analysis (Munich), 29 (2009) 407--430.

\bibitem{botsch} M. Botsch, L. Kobbelt, M. Pauly, P. Alliez, and B. L\'evy, \textit{Polygon Mesh Processing}, 2010,  AK Peters.

\bibitem{bryant2}
R. Bryant,
\textit{A duality theorem for Willmore surfaces.}
J. Differential Geom. 20 (1984), 23Ð53. 

\bibitem{bryant}
R. Bryant, \textit{Surfaces in conformal geometry,} 
 The mathematical heritage of Hermann Weyl, Proc. Sympos. Pure Math., vol. 48 (1988), pp. 227--240

\bibitem{chen}
B.-Y. Chen, 
\textit{On the total curvature of immersed manifolds, III. Surfaces in Euclidean $4$-space,}
Amer. J. Math 95 (1973), 636--642. 

\bibitem{chen2}
B.-Y. Chen, 
\textit{On the total curvature of immersed manifolds. V. C-surfaces in Euclidean m-space,}
Bull. Inst. Math. Acad. Sinica 9 (1981), 509--516. 


\bibitem{chern} S.-S. Chern, R. Lashof, 
\textit{On the total curvature of immersed manifolds. II,} 
Michigan Math. J. 5 (1958) 5--12. . 

\bibitem{colding-delellis} 
T. Colding and C. De Lellis, \textit{The min-max construction of minimal surfaces,}   Surveys in  Differential Geometry VIII , International Press,  (2003),   75--107.




\bibitem{ejiri}E. Ejiri, \textit{Willmore surfaces with a duality in $S^N(1)$}, Proc. London Math. Soc., 57 (1988), 383--416.

\bibitem{fary} I. F\'ary, 
\textit{Sur la courbure totale d'une courbe gauche faisant un nÏud},
Bull. Soc. Math. France 77, (1949) 128--138.

\bibitem{fenchel}
W. Fenchel,
\textit{\"{U}ber Kr\"{u}mmung und Windung geschlossener Raumkurven},
Math. Ann. 101 (1929), no. 1, 238--252.  


\bibitem{ferus-pedit} D. Ferus and F. Pedit, 
\textit{$S^1$-equivariant minimal tori in $S^4$ and $S^1$-equivariant Willmore tori in $S^3$},
Math. Z. 204 (1990), 269--282. 

\bibitem{francis} G. Francis, J. Sullivan, R. Kusner, K. Brakke, C. Hartman, and G. Chappell, \textit{The minimax sphere eversion,} In Hege and Polth- ier, editors, Visualization and Mathematics, pages 3--20. Springer, 1997.


\bibitem{germain}
S. Germain, \textit{Recherches sur la th\'eorie des surfaces \'elastiques,} Paris (1921).



\bibitem{heintze-karcher}
E. Heintze and H. Karcher,
\textit{A general comparison theorem with applications to volume estimates for submanifolds,}
Annales  scientifiques de l'E.N.S. 11 (1978), 451--470.

\bibitem{helfrich}
W. Helfrich.
\textit{Elastic properties of lipid bilayers: Theory and possible experiments.}
Z. Naturforsch.  28 (1973), 693--703. 

\bibitem{pinkall2} U. Hertrich-Jeromin and U. Pinkall,
{\em Ein Beweis der Willmoreschen Vermutung fŸr Kanaltori.} 
J. Reine Angew. Math. 430 (1992), 21--34. 

\bibitem{hsu-kusner-sullivan} L. Hsu, R.  Kusner,  and J. Sullivan, 
\textit{Minimizing the squared mean curvature integral for surfaces in space forms,} 
Experiment. Math. 1 (1992), 191--207.


\bibitem{kuhnel-pinkall}
W. K\"uhnel and U. Pinkall, 
\textit{On total mean curvatures,}
Quart. J. Math. Oxford Ser. (2) 37 (1986), 437--447. 

\bibitem{kusner}
R. Kusner,
\textit{Conformal geometry and complete minimal surfaces,}
Bull. Amer. Math. Soc. (N.S.) 17 (1987), 291--295. 

\bibitem{kusner89}
R. Kusner, 
\textit{Comparison surfaces for the Willmore problem,}
Pacific J. Math. 138 (1989), 317--345. 

\bibitem{kusner96}
R. Kusner, 
\textit{Estimates for the biharmonic energy on unbounded planar domains, and the existence of surfaces of every genus that minimize the squared-mean-curvature integral,}  Elliptic and parabolic methods in geometry, A K Peters, (1996), 67--72. 

\bibitem{kuwert-schatzle}
E. Kuwert and R. Sch\"atzle,
\textit{Removability of point singularities of Willmore surfaces,}
Ann. of Math. 160 (2004), 315--357. 


\bibitem{kuwert-li-schatzle}
E. Kuwert, Y. Li, and R.  Sch\"{a}tzle,
\textit{The large genus limit of the infimum of the Willmore energy,}
Amer. J. Math. 132 (2010), 37--51. 

\bibitem{kuwert-schatzle-survey}
E. Kuwert and R. Sch\"{a}tzle, 
\textit{The Willmore functional,} Topics in modern regularity theory, 1--115,
CRM Series, 13, Ed. Norm., Pisa, 2012. 

\bibitem{lamm} T. Lamm and H.T. Nguyen, \textit{Branched Willmore Spheres}, to appear in J. Reine Angew. Math.

\bibitem{langer-singer}
J. Langer and D.  Singer,
\textit{Curves in the hyperbolic plane and mean curvature of tori in 3-space,}
Bull. London Math. Soc. 16 (1984), 531--534.

\bibitem{langer-singer2}
J. Langer and D.  Singer, \textit{
Curve-straightening in Riemannian manifolds}, 
Ann. Global Anal. Geom. 5 (1987), 133--150. 

\bibitem{langevin-rosenberg}
R. Langevin and H. Rosenberg,
\textit{On curvature integrals and knots,}
Topology 15 (1976), 405--416. 

\bibitem{lawson70}
B. Lawson, \textit{Complete minimal surfaces in $S^3$,}
Ann. of Math. (2) 92 (1970), 335--374. 

\bibitem{li-yau}
P. Li and S-T. Yau,
\textit{A new conformal invariant and its applications to the Willmore conjecture and the first eigenvalue of compact surfaces,}
Invent. Math. 69 (1982), 269--291. 



\bibitem{lott} N. Lott, D. Pullin, \textit{Method for fairing B-spline surfaces}, Comput-Aided Des 20 (1988), no. 10. 597--600

 

\bibitem{marques-neves} F. C. Marques and A. Neves,
\textit{Min-max theory and the Willmore conjecture,}
Ann. of Math. (2) 179 (2014), no. 2, 683--782. 



\bibitem{mayer} U. F. Mayer and G. Simonett,  \textit{A numerical scheme for axisymmetric solutions of curvature-
driven free boundary problems, with applications to the Willmore flow}, Interfaces Free
Bound., 4 (2002), 89--109.

\bibitem{bensimon2} X.  Michalet and D. Bensimon, \textit{Vesicles of toroidal topology: observed morphology and shape transformations.} J. Phys. II France 5 (1995), 263--287.

\bibitem{milnor}
J. W. Milnor, 
\textit{On the total curvature of knots},
Ann. of Math. (2) 52, (1950). 248--257. 

\bibitem{montiel} S. Montiel, \textit{Willmore two-spheres in the four sphere}, Trans. Amer. Math. Soc., (2000), 4449--
4486

\bibitem{montiel-ros}
S. Montiel and A. Ros, 
\textit{Minimal immersions of surfaces by the first eigenfunctions and conformal area,}
Invent. Math. 83 (1986), 153--166.

\bibitem{montiel-urbano} S. Montiel and F. Urbano, \textit{A Willmore functional for compact surfaces in the complex projective plane}, J. Reine Angew. Math., 546 (2002), 139--154.


\bibitem{musso} E. Musso, \textit{Willmore surfaces in the four-sphere}, Ann. Global Anal. Geom., 8 (1990), 21--41.

\bibitem{bensimon}

 M. Mutz and D. Bensimon, \textit{Observation of toroidal vesicles.} Phys. Rev. A, 43 (1991), 4525--4527.


\bibitem{pitts} 
J. Pitts, \textit{Existence and regularity of minimal surfaces on Riemannian manifolds,} Mathematical Notes 27, Princeton University Press, Princeton, (1981).

\bibitem{poisson}
S. D. Poisson, \textit{M\'emoire sur les surfaces \'elastiques,} Mem. Cl. Sci. Math. Phys., Inst. de France, (1812) 167--225.


\bibitem{peng}
C.--K. Peng, and L. Xiao, \textit{
Willmore surfaces and minimal surfaces with flat ends},  Geometry and topology of submanifolds, X (Beijing/Berlin, 1999), 259--265, World Sci. Publ., River Edge, NJ, (2000).

\bibitem{pinkall}
U. Pinkall, \textit{Hopf tori in $S^3$,}
Invent. Math. 81 (1985), 379--386. 

\bibitem{ritore-ros}
M. Ritor\'{e} and A. Ros, 
\textit{Stable constant mean curvature tori and the isoperimetric problem in three space forms.}
Comment. Math. Helv. 67 (1992), no. 2, 293Ð305. 

\bibitem{riviere}
T. Rivi\`ere, 
\textit{Analysis aspects of Willmore surfaces,}
Invent. Math. 174 (2008) 1--45. 

\bibitem{riviere-notes}
T. Rivi\`ere, 
\textit{Conformally invariant variational problems,}
arXiv:1206.2116 [math.AP] (2012).

\bibitem{ros}
A. Ros, \textit{The Willmore conjecture in the real projective space,}
Math. Res. Lett. 6 (1999), 487--493.

\bibitem{ros2001}
A. Ros,
\textit{The isoperimetric and Willmore problems,}
Global differential geometry: the mathematical legacy of Alfred Gray (Bilbao, 2000), 149--161,
Contemp. Math., 288, Amer. Math. Soc., Providence, RI, (2001). 



\bibitem{shiohama-takagi}
K. Shiohama and R. Takagi,
\textit{A characterization of a standard torus in $E^3$,}
J. Differential Geometry 4 1970 477--485. 


\bibitem{simon93}
L. Simon, \textit{Existence of surfaces minimizing the Willmore functional,}
Comm. Anal. Geom. 1 (1993), no. 2, 281--326. 

\bibitem{smale-eversion}
S. Smale, 
\textit{A classification of immersions of the two-sphere,}
Trans. Amer. Math. Soc. 90 (1958) 281--290.



\bibitem{thomsen}
G. Thomsen, \textit{\"{U}ber Konforme Geometrie, I: Grundlagen der Konformen Fl\"achentheorie,} Abh. Math. 
Sem. Hamburg (1924), 31--56.

\bibitem{topping}
P. Topping, 
\textit{Towards the Willmore conjecture,}
Calc. Var. Partial Differential Equations 11 (2000), 361--393. 

\bibitem{urbano}
F. Urbano,
\textit{Minimal surfaces with low index in the three-dimensional sphere,}
Proc. Amer. Math. Soc. 108 (1990), 989--992.


\bibitem{whitej}
J. H. White, \textit{A global invariant of conformal mappings in space,}
Proc. Amer. Math. Soc. 38 (1973), 162--164. 

\bibitem{weiner}
J. L. Weiner,
\textit{On a problem of Chen, Willmore, et al,}
Indiana Univ. Math. J. 27 (1978), 19--35. 

\bibitem{willmore}
T. J. Willmore, 
\textit{Note on embedded surfaces,} An. Sti. Univ.``Al. I. Cuza'' Iasi Sect. I a Mat. (N.S.) 11B (1965) 493--496. 

\bibitem{willmore71}
T. J. Willmore,
\textit{Mean curvature of Riemannian immersions,}
J. London Math. Soc. (2) 3 1971 307--310. 


\end{thebibliography}

\end{document}